\documentclass[reqno,12pt]{amsart}
\usepackage[utf8]{inputenc}
\usepackage{enumerate}
\usepackage{amssymb,mathrsfs,latexsym,verbatim,pdfsync,color,mathabx}
\usepackage{upref}
\usepackage{dutchcal}
\usepackage{color}
\usepackage[margin=2cm]{geometry}
\usepackage{tikz-cd}
\usepackage[colorlinks=true, linkcolor=blue]{hyperref}
\usepackage{centernot}
\usepackage{mathtools}
\usepackage{ stmaryrd }

\newcommand{\NN}{\mathbb{N}}

\newcommand{\cC}{{\mathcal{C}}}
\newcommand{\cD}{{\mathcal{D}}}

\newcommand{\cM}{{\mathcal{M}}}

\newcommand{\cB}{{\mathcal{B}}}
\newcommand{\cF}{\mathcal{F}}

\newcommand{\norm}{\Vert}
\newcommand{\bea}{\begin{align}}
\newcommand{\eea}{\end{align}}
\newcommand{\beqa}{\begin{align*}}
\newcommand{\eeqa}{\end{align*}}

\newcommand{\inner}[2]{\langle #1 , #2 \rangle}
\newcommand{\ssl}{\sigma_\lambda}

\DeclareMathOperator{\rank}{rank}

\DeclareMathOperator{\tr}{tr}
\DeclareMathOperator{\Span}{span}
\DeclareMathOperator{\Id}{Id}

\DeclareMathOperator{\supp}{supp}
\newtheorem{theorem}{Theorem}[section]
 \newtheorem{lemma}[theorem]{Lemma}

\theoremstyle{definition}
\newtheorem{definition}[theorem]{Definition}

\theoremstyle{remark}

\numberwithin{equation}{section}

\newcommand{\bbB}{{\mathbb B}}
\newcommand{\bbC}{{\mathbb C}}
\newcommand{\bbD}{{\mathbb D}}
\newcommand{\bbH}{{\mathbb H}} 
\newcommand{\bbN}{{\mathbb N}}
\newcommand{\bbR}{{\mathbb R}}

\def\bU{{\mathbf U}}

\def\cB{{\mathcal B}}
\def\cC{{\mathcal C}}
\def\cD{{\mathcal D}}

\def\cF{{\mathcal F}}

\def\cH{{\mathcal H}}

\def\cL{{\mathcal L}}
\def\cM{{\mathcal M}}

\def\cU{{\mathcal U}}

\def\sS{{\mathscr S}}

\def\Re{\operatorname{Re}}
\def\Im{\operatorname{Im}}

\def\z{\zeta} 

\def\ov{\overline}
\def\p{\partial}

\def\ms{\medskip}

\def\Hol{\operatorname{Hol}}
\def\tr{\operatorname{tr}}

\def\HS{\operatorname{HS}}
\def\ran{\operatorname{ran}}
\def\Dom{\operatorname{Dom}}

\def\half{\textstyle{\frac12}}

 \address{Dipartimento di Matematica, Alma Mater Studorium Universit\`a di
  Bologna, Piazza di Porta San Donato 5, 40126 Bologna, Italy}
  \email{{\tt nicola.arcozzi@unibo.it}}
 \address{Dipartimento di Matematica, Alma Mater Studorium Universit\`a di
  Bologna, Piazza di Porta San Donato 5, 40126 Bologna, Italy}
  \email{{\tt nikolaos.chalmoukis2@unibo.it}}
    \address{Dipartimento di Matematica e Applicazioni, Universit\`a degli Studi di
  Milano--Bicocca, Via R. Cozzi 55, 20126 Milano, Italy}
\email{{\tt alessandro.monguzzi@unimib.it}}
\address{Dipartimento di Matematica, Universit\`a degli Studi di
  Milano, Via C. Saldini 50, 20133 Milano, Italy}
  \email{{\tt marco.peloso@unimi.it}}
  \address{Dipartimento di Matematica, Universit\`a degli Studi di
  Milano, Via C. Saldini 50, 20133 Milano, Italy}
\email{{\tt maura.salvatori@unimi.it}}
\medskip

\thanks{{\em Math Subject Classification 2020:} 47A13, 47B37, 46E22 .}
\thanks{All the authors are members of INdAM.  The first and
  third author are partially supported by the Progetto GNAMPA
  2020 {\em Alla
frontiera tra l’analisi complessa in pi\`u variabili e l’analisi
armonica}. The
  fourth and fifth author are  partially
  supported by the Progetto GNAMPA 2020 \emph{Fractional
    Laplacians and subLaplacians on Lie groups and trees}. }

\title[The Drury--Arveson space and a von Neumann type inequality]{The Drury--Arveson space on the Siegel upper half-space and a von Neumann type inequality}

\author[N. Arcozzi, N. Chalmoukis, A. Monguzzi, M. M. Peloso, M. Salvatori]{Nicola
  Arcozzi, Nikolaos Chalmoukis,
Alessandro Monguzzi, Marco M. Peloso and Maura Salvatori}
\date{}

\begin{document}

\begin{abstract}
 In this work we study what we call Siegel--dissipative vector of
 commuting operators $(A_1,\ldots, A_{d+1})$ on a Hilbert space $\cH$
 and we obtain a von Neumann type inequality which involves the
 Drury--Arveson space $DA$ on the Siegel upper half-space
 $\cU$. The operator $A_{d+1}$ is allowed to be unbounded and it
 is the infinitesimal generator of a contraction semigroup
 $\{e^{-i\tau A_{d+1}}\}_{\tau<0}$. We then study the operator
 $e^{-i\tau A_{d+1}}A^{\alpha}$ where $A^{\alpha}=A_1^{\alpha_1}\cdots
 A^{\alpha_d}_d$ for $\alpha\in\bbN_0^d$ and prove that can be studied
 by means of  model operators on a weighted $L^2$ space.  
 To prove our results we obtain a Paley--Wiener type theorem for $DA$ and we investigate some multiplier operators on $DA$ as well.
\end{abstract}
\maketitle

\section{Introduction}
The Drury--Arveson space $DA(\bbB^{d+1})$ on the unit ball of $\bbC^{d+1}$ is a
renowned  Hilbert space of holomorphic functions endowed with the inner product
$$
\langle \sum_{\alpha\in\bbN_0^{d+1}}b_\alpha
z^\alpha,\sum_{\alpha\in\bbN_0^{d+1}}c_\alpha z^\alpha\rangle_{DA(\bbB^{d+1})}
:=\sum_{\alpha\in\bbN_0^{d+1}}b_\alpha\overline{c_\alpha}\frac{\alpha!}{|\alpha|!}.
$$
We refer the reader to \cite[Theorem 6.1]{AMPS} for an exact integral
representation of such inner product. The space $DA(\bbB^{d+1})$
is a reproducing kernel Hilbert space with kernel
$K(z,w)=(1-z\cdot\overline w)^{-1}$ and may be considered \emph{the} natural multi-dimensional version of the Hardy space on the
unit disc $H^2(\bbD)$;  one of the main reasons 
being that $DA(\bbB^{d+1})$ plays the same role as $H^2(\bbD)$ in a multi-dimensional
version of the famous von Neumann Inequality. 
\begin{theorem}[Drury \cite{Drury}]\label{D-vN-thm}
Let $\cH$ be a Hilbert space and consider the $(d+1)$-tuple
$T=(T_1,\ldots, T_{d+1})$ of linear operators on $\cH$
satisfying
\begin{enumerate}
    \item[(i)] $T_jT_k-T_kT_j=0$ for all $j,k=1,\ldots,d+1$;
    \item[(ii)] $\sum_{j=1}^{d+1}\|T_j v\|_{\cH}^2\leq \|v\|^2_{\cH}$ for all $v\in\cH$.
\end{enumerate}
Let $p(z)=p(z_1,\ldots,z_{d+1})$ be a complex polynomial. Then,
$$
\|p(T)\|_{_{\cB(\cH)}}\leq \|p\|_{\cM(DA(\bbB^{d+1}))}
$$
where $\|\cdot\|_{\cB(\cH)}$ and $\|\cdot\|_{\cM(DA(\bbB^{d+1}))}$ denote the
norm of bounded linear operators on $\cH$ and the multiplier operator
norm on the Drury--Arveson space on the unit ball respectively. 
\end{theorem}
In the case $d=0$ Drury's result reduces to the classical von Neumann Inequality (\cite{vonNeumann}) and the multiplier operator norm on $H^2(\bbD)$ takes the place of $\|\cdot\|_{\cM(DA(\mathbb B^{d+1}))}$. Namely, given a Hilbert space $\cH$ and $T:\cH\to \cH$ a contraction, i.e., $\|T\|_{\cB(\cH)}\leq 1$, the von Neumann Inequality states that, for any polynomial $p(z)$, 
\begin{equation*}
\|p(T)\|_{\cB(\cH)}\leq \|p\|_{\cM(H^2(\bbD))}=\|p\|_{\infty}.
\end{equation*}

It is possible to prove an 
analogous result modeled on the upper-half plane, which is the
unbounded biholomorphic realization of the unit disk $\bbD$ via the Cayley transform.
Let $A$ be any 
 bounded dissipative operator on
a Hilbert space  $\cH$, that is, such  that $\frac1i (A-A^*)\ge0$,
then 
\begin{equation*}
   \|f(A)\|_{\cB(\cH)}\leq \sup_{\Im z>0}|f(z)|
\end{equation*}
for any  rational function $f$ which is bounded on the upper
half-plane, see \cite{Waksman}.   

The main goal of this paper is to prove a version of Theorem \ref{D-vN-thm} on an unbounded biholomorphic realization of
the unit ball $\bbB^{d+1}$ in $\bbC^{d+1}$, that is,
 the Siegel upper-half space $\cU$, 
\begin{equation}\label{Siegel-domain}
\cU=\left\{(\z,\z_{d+1}):\bbC^{d+1}: \Im\z_{d+1}> \textstyle{\frac14}|\z|^2\right\}.
\end{equation}
The biholomorphism between the ball and the Siegel half-space is given by the multi-dimensional Cayley transform
$\cC:\mathbb B^{d+1}\to \cU$, 
\begin{equation}\label{cayley_transform}
\mathcal C(\omega,\omega_{d+1})=\bigg(\frac{2\omega}{1-\omega_{d+1}},i\frac{1+\omega_{d+1}}{1-\omega_{d+1}}\bigg).
\end{equation}
The Drury--Arveson space on $\cU$, that we shall simply denote by $DA$, was studied
in \cite{ARS} and \cite{AMPS}, where an integral expression for the norm was also obtained.
Let $\rho(\z,\z_{d+1})=\Im\z_{d+1}-\frac14|\z|^2$ and let $n$ be an integer such that $n>d/2$. Then, we define the space $DA_{(n)}$ as
\begin{align}\label{DA-m-def} \nonumber
 DA_{(n)}=\Big\{ 
F\in\Hol(\cU) :&
{\rm (i)}  \displaystyle { \lim_{|\z|\le R,\, \Im\z_{d+1}\to+\infty} 
  F(\z,\z_{d+1}) =0} \,;& \smallskip  \\
&   {\rm (ii)}   \displaystyle {\int_\cU |\rho^n(\z,\z_{d+1}) \p_{\z_{d+1}}^n F(\z,\z_{d+1})|^2 \,
\rho^{-d-1}(\z,\z_{d+1}) d\z d\z_{d+1} <+\infty  }\, \Big\}\,. 
\end{align}
The space $DA_{(n)}$ does not actually depend on $n$, so that we simply write $DA$ in place of $DA_{(n)}$. We refer the reader to  Section
\ref{DA-section} for this
and other properties of the space $DA$.

We now describe the operator-theorical condition we shall work with.
\begin{definition}
Let $\cH$ be a Hilbert space and let $(A_1,\ldots, A_{d+1})$ be a
vector of operators on $\cH$ and let $r(iA_{d+1})$ be the resolvent set of $iA_{d+1}$. 
We say that $(A_1,\ldots,A_{d+1})$ is a \emph{Siegel--dissipative vector of commuting operators} if:
\begin{enumerate}
\item[(i)] the operators $A_1,\dots A_d$ are bounded and $(0,+\infty) \subseteq r (i A_{d+1})$;
  \item[(ii)] the operators $A_1, \dots A_d$ commute with each other and they strongly commute with $A_{d+1}$;
\item[(iii)]  the following condition  holds:\footnote{The multiplicative constant $\textstyle{\frac{1}4}$ in
\eqref{Siegel-Inequality} appears because of the definition of $\cU$ we use. The
presence of this constant might be unpleasant but we keep this
definition to be consistent  with the definition in \cite{AMPS}.}
\begin{equation} \label{Siegel-Inequality}
  \Im\inner{A_{d+1}v}{v}_{\cH}\geq \frac14\sum_{i=1}^d \norm{A_i v}\norm^2_{\cH}, \,\,\, \forall v \in \Dom(A_{d+1}).
  \end{equation} 
  \end{enumerate}
\end{definition}
Given two operators $U$ and $T$ 
on $\cH$, $U$ densely defined and closed,  $T$ bounded,  we say that $U$ and $T$ strongly commute if $UT$ is an extension of $TU$. In other words  $T(\cH)
\subseteq \Dom(U)$  and $TUv=UTv$ for all $v \in
\Dom(U)$. 

We point out that from the conditions above and the Lumer--Philips Theorem (Theorem \ref{Lumer-Philips}) it follows that $iA_{d+1}$
is the infinitesimal generator of a (unique) semigroup of contractions
$\!\{e^{-i\tau A_{d+1}}\!\}_{\tau\leq0}$ that commutes with all the
$A_j$'s, see Lemma \ref{commutation}.
 Our von Neumann type inequality reads as follows.
\begin{theorem} \label{thm-vonNeumann}
Let $\cH$ be a Hilbert space and let $(A_1,\ldots,A_{d+1})$ be a
Siegel--dissipative vector of commuting operators.  For any
$\tau_j<0$, $j=1,\dots,d$, set $M_j = e^{-i\tau_j
  A_{d+1}}A_j$ and $m_j(\zeta,\zeta_{d+1}) =  e^{-i\tau_j
  \zeta_{d+1}}\zeta_j$. 
Let $p$ denote any polynomial in $d$ variables.
Then we have that $m_j\in\cM(DA)$ for $j=1,\dots,d$, and
\begin{equation}\label{main-estimate}
  \| p\big( M_1,\dots, M_d)\|_{\cB(\cH)}\leq
  \| p\big( m_1,\dots, m_d)  \|_{\cM(DA)} ,
\end{equation}
where $\cM(DA)$ denotes the space of pointwise multipliers on the
Drury--Arveson space $DA$.  
\end{theorem}

A few remarks are now in order. The space $DA$ is
 a reproducing kernel Hilbert space, and, with respect to the $DA_{(n)}$-norm, its reproducing kernel is
\begin{equation*}
K_{DA}((\omega,\omega_{d+1}),(\z,\z_{d+1)}) = \gamma_{d,n} 
\Big( \frac{\omega_{d+1}-\ov\z_{d+1}}{2i} -
\textstyle{\frac{1}{4}} \omega\cdot \ov{\z} 
\Big)^{-1} \,,
\end{equation*}
where $\gamma_{d,n}=
\frac{4^n}{(4\pi)^{d+1}\Gamma(2n-d)}$.
Hence, notice that the condition \eqref{Siegel-Inequality} is modelled after
 the reproducing kernel of $DA$. In general, it is a well
 established phenomenon that von Neumann type inequalities hold as
 long as one can suitably interpret the condition $K^{-1}(A,A)\geq 0$,
 where $K$ is the reproducing kernel of a space of holomorphic
 functions in some domain in $\bbC^d$. A precise formulation of this
 result is stated and proved in \cite{Ambrozie}. There the authors use
 the Dunford-Riesz functional calculus to define $K^{-1}(A, A)$. This
 is possible since they are working with bounded
 operators with some additional assumptions on their spectrum. Furthermore, they assume that the multiplication by any
 coordinate function is a continuous (bounded) operation. Our setting
 differs in these two latter aspects. The last operator $A_{d+1}$ is
 allowed to be unbounded and also the multiplication by any coordinate
 is an unbounded operation in $DA$. 

We also mention that, in the case when all the operators are 
bounded,  a proof of the von Neumann type inequality
\eqref{main-estimate} could probably be obtained by means of the Cayley transform,
the classical inequality of Drury on the unit ball and some classical
results in operator theory as it was pointed out to us by M. Hartz in
a private communication \cite{Hartz}. Nevertheless, in the footsteps of Drury's proof, we prefer to
use a direct approach, without relying on the known results on the unit ball, for two main reasons.
First, we have greater generality by allowing one operator to be
unbounded. Second, and more importantly, we 
 develop some machinery we believe it is interesting
in its own right and has the potential to be used in the context of
the theory of shift-invariant subspaces of $DA$, as in the spirit of
Buerling--Lax's Theorem \cite{Lax-Acta}, 
in the Siegel half-space setting.

In order to follow such a plan, we further investigate some
operator-theorical properties of $DA$.   
On the space $\bbN_0^d\times\bbR_-$ we consider the measure
\begin{equation}\label{measure-mu}
d\mu(\alpha,\lambda) 
=\alpha!\left(\frac 2{|\lambda|}\right)^{|\alpha|}|\lambda|^{2d}\, d\alpha\, d\lambda
\end{equation}
with $d\lambda$ denoting the Lebesgue measure on $\bbR_{-}$ and $d\alpha$ the counting measure on $\bbN_0^{d}$.
\begin{theorem}\label{isometry-Phi}
The map $\sS:\, L^2(\bbN_0^d\times\bbR_-,\,d\mu)\to DA$ defined as 
\begin{equation*}
(\sS\varphi) (\zeta,\zeta_{n+1})   =  \frac{1}{(2\pi)^{d+1}} \int_{\bbN_0^d\times\bbR_-}  \z^\alpha e^{-i\lambda\z_{d+1}} \overline{\varphi(\alpha,\lambda)}\,|\lambda|^d\, d\alpha d\lambda
\end{equation*}
is a conjugate linear surjective isometry. The inverse map $\sS^{-1}$ is explicitly given in \eqref{map-Phi}. 
\end{theorem}

 If for some $\varepsilon >0 $ the
  $(d+1)$-tuple $(A_1,\dots,A_d,A_{d+1}-i\varepsilon \Id)$ is
  Siegel-dissipative,  we say that $(A_1,\ldots,A_{d+1})$
  is \emph{strongly} Siegel--dissipative.   In this case, 
for $v,v'\in\Dom(A_{d+1})$ consider the inner product
$$
\langle v,v'\rangle_{\Delta}:=\frac{1}{2i}\big(\langle
A_{d+1}v,v'\rangle_{\cH}-\langle
v,A_{d+1}v'\rangle_{\cH}\big)-\frac{1}{4}\sum_{i=1}^{d}\langle
v,v'\rangle_{\cH} 
$$
so that 
$$
\norm v \norm^2 _{\Delta}  =\Im\inner{A_{d+1}v}{v}_{\cH}-\frac14 \sum_{i=1}^d \norm A_i v \norm ^2_{\cH} \geq \varepsilon \norm v \norm ^2_{\cH}, \,\,\, \forall v
\in \Dom(A_{d+1}),
$$  
which is a stronger condition than
\eqref{Siegel-Inequality}. If $\cH_\Delta=(\Dom{A_{d+1}},\|\cdot\|_{\Delta})$ we define 
\begin{align*}
\begin{split}
 \cL^2&(\Delta):=L^2(\bbN_0^d\times\bbR_{-},d\mu;\cH_\Delta)\\
 &=\Big\{ g :\bbN_0^d\times\bbR_{-} \to \Dom A_{d+1}: \sum_{\alpha
   \in \bbN_0^d}\int_{-\infty}^0 \norm g (\alpha,\lambda)\norm^2_\Delta \ d\mu(\alpha,\lambda)<+\infty\Big\}.
 \end{split}
\end{align*}

With our assumptions it follows that $\cL^2(\Delta)$ is a pre-Hilbert  space.
Before stating our result on $d+1$-tuples of operators we also need to study some weighted shift operators on $L^2(\bbN_0^d\times\bbR_-,\, d\mu)$, which correspond to some multiplier operators on $DA$ (see Section \ref{multipliers-section}). If $\alpha,\gamma\in \bbN_0^d$ we write $\alpha\geq\gamma$ meaning that $\alpha_i\geq \gamma_i$ for all $i=1,\ldots,d$. Then, we prove the following. 
 \begin{theorem}\label{shift-adjoint}
 Let $\gamma \in \bbN_0^d$ and $\tau<0 $. Then, the operator 
 \[
     (S_{\gamma,\tau} \varphi)(\alpha,\lambda) = \begin{cases}  \frac{|\lambda-\tau|^d}{|\lambda|^d}\varphi(\alpha - \gamma, \lambda - \tau) & \text{if} \quad \lambda<\tau \wedge \alpha\geq\gamma;
     \\
     \phantom{phantom}\\
     0 & \text{otherwise,}
     \end{cases}
 \] extends to a bounded linear operator on $L^2(\bbN_0^d\times \bbR_-,
 d\mu )$ unitarily equivalent to the multiplier operator on $DA$
 with multiplier $\z^\gamma e^{-i\tau \z_{d+1}} $. Furthermore, its adjoint is given by the formula 
 \[ (S_{\gamma,\tau}^*\varphi)(\alpha, \lambda) = \frac{|\lambda+\tau|^{d-|\alpha|-|\gamma|}}{|\lambda|^{d-|\alpha|}} \frac{(\alpha+\gamma)!}{\alpha ! } 2^{|\gamma|} \varphi(\alpha+\gamma, \lambda + \tau ).\]
 \end{theorem}

We remark that, if we call the multipliers
$m_j$, $j=1,\dots,d$, that appear in Theorem \ref{thm-vonNeumann} the
{\em shift operators} on $DA$, then the operators $S_{\gamma,\tau}$
correspond to the operators $\sS^{-1} m_j \sS$. 
Finally, we have the following result.
\begin{theorem}\label{lifting-thm}
Let $(A_1,\ldots,A_{d+1})$ be a strongly Siegel--dissipative vector of
commuting operators and set $A^{\alpha}= A_1^{\alpha_1}\cdots A^{\alpha_d}_d$ for any multi-index $\alpha\in\bbN_0^d$. Then:
\begin{enumerate}
    \item[(i)] the map $\Theta:\cH\to L^2(\NN_0^d\times\bbR_{-},d\mu ; \cH_\Delta)$ defined as
\begin{equation*}
(\Theta v)(\alpha,\lambda)=\frac{|\lambda|^{|\alpha|-d}}{\alpha! 2^{|\alpha|-\frac{1}{2}}}e^{-i\lambda A_{d+1}}A^\alpha v
\end{equation*}
is an isometric embedding, i.e.,  \begin{equation*}
  \|\Theta v\|_{\cL^2(\Delta)}=\|v\|_{\cH}, \,\,\, \forall v \in \cH .
 \end{equation*}
\item[(ii)] the diagram 
\begin{equation}\label{diagram-Siegel}
\begin{tikzcd}
 L^2(\NN_0^d\times\bbR_{-},d\mu ; \cH_\Delta) \arrow{r}{S_{\gamma,\tau}^* \otimes \Id }  &  L^2(\NN_0^d\times\bbR_{-},d\mu ; \cH_\Delta) \\
\arrow{u}{\Theta} \cH \arrow{r}{ e^{-i\tau A_{d+1}}A^\gamma } & \cH \arrow{u}{\Theta}
\end{tikzcd}
\end{equation}
commutes.
\end{enumerate}
\end{theorem}
We conclude this introduction pointing out that the Drury--Arveson
space $DA(\bbB^{d+1})$ on the unit ball of $\bbC^{d+1}$ has drawn considerable
interest since its first appearance in the works of Drury and of Arveson
\cite{Drury}, \cite{Arveson}.   We mention \cite{ARS_carleson,ARS,AL},
\cite{fang_xia}, \cite{Richter-Sunkes} and references therein.  The
Drury--Arveson space on $\cU$ was studied
in \cite{ARS} and\cite{AMPS}; see also \cite{CP} for the case of more general Siegel type domains.

The paper is organized as follows. In Section \ref{resti}
 we recall some preliminary facts,  in Section \ref{DA-section} we
 study the space $DA$ and we prove Theorem \ref{isometry-Phi}. In Section \ref{multipliers-section} we study some multiplier operators on $DA$ and we prove Theorem \ref{shift-adjoint}. We conclude proving  Theorem \ref{lifting-thm} and Theorem \ref{thm-vonNeumann} in Section \ref{lifting-section}.

\section{Preliminary facts}\label{resti}

In this section we first recall some basic facts on the Siegel
half-space, the Heisenberg group and the group Fourier
transform. Then we also recall some simple properties of semigroups of
operators on a Hilbert space.

\subsection{The Siegel upper half-space and the Heisenberg group}

The Siegel upper half-space is defined in 
\eqref{Siegel-domain} and it is biholomorphically equivalent to the unit ball
$\bbB^{d+1}$ via the Cayley transfom \eqref{cayley_transform}.
Standard references for the facts that follow are e.g. \cite{Folland-phase,Stein,
  Thangavelu}.
    We
    introduce new coordinates on $\cU$ setting $\Psi(\z,\z_{d+1})=(z,t,h)$ where
    \begin{equation*}
    \begin{cases}
    z=\z\cr
    t= \Re \z_{d+1}\cr
    h =\Im \z_{d+1} -\frac14|\z|^2 \,.
    \end{cases}
    \end{equation*}
    Then,  if $\bU =\bbC^d\times\bbR\times(0,+\infty)$, 
    $\Psi: \ov\cU\to
    \ov\bU$ is a $C^\infty$-diffeomorphism,  
    and $\Psi^{-1}$
    is given by
    \begin{equation*}    \Psi^{-1} (z,t,h) = \bigl(z,t+i\textstyle{\frac14}|z|^2+ih\bigr) =:  (\z,\z_{d+1})
    \,. 
\end{equation*}
Notice that
$h=\varrho(\z,\z_{d+1})$ and $\Psi(\z, t+\frac i4|\z|^2)=(z,t,0)$. In
this case we write $[z,t]$ in place of $(z,t,0)$ and we also use the abuse of notation $[z,t]\in \p \cU$. We now let the
points in $\p\cU$ act on $\ov\cU$ as biholomorphic maps in the
following way. For $[z,t]\in\p\cU$ set 
\begin{equation}\label{Phi}
\Phi_{[z,t]} (\omega,\omega_{d+1})
= \bigl(\omega+z, \omega_{d+1} +t
+i\textstyle{\frac14}|z|^2+ \frac i2 \omega\cdot\bar z \bigr) \,,
\end{equation}
where $\omega\cdot \bar z=\sum_{j=1}^d\omega_j\bar z_j$ denotes the inner product in $\bbC^d$. 
Notice that
\begin{align*}
\varrho\Bigl(
\Phi_{[z,t]}
(\omega,\omega_{d+1})\Bigr) 
& = \varrho(\omega,\omega_{d+1})\,,  
\end{align*}
that is, the maps $\Phi_{[z,t]}$ preserve the
defining function $\varrho$. 
 In particular, for $(\omega,\omega_{d+1})\in\p\cU$ and
$[w,s]=\Psi(\omega,\omega_{d+1})$, by \eqref{Phi} we have
\begin{align*}
\Phi_{[z,t]} \big(
(\omega,\omega_{d+1})\big) & = \Phi_{[z,t]} \Big( \Psi^{-1}
(w,s,0)\Big) = 
\Phi_{[z,t]} \bigl(w,s+i\textstyle{\frac14}|w|^2\bigr) \notag \\
 & =\bigl(w+z,s+\textstyle{\frac i4}|w|^2+t+ \textstyle{\frac i4}|z|^2
+\textstyle{\frac i2} w\cdot\bar
z\bigr) \notag \\
& =\bigl[ w+z,s+t -\half\Im(w \cdot\bar
z)\bigr] \notag \\ 
& =: [w,s][z,t] \, .  
\end{align*}
Hence, it is possible to introduce a group structure on $\p\cU$ itself.
\begin{definition}{\rm
The Heisenberg group $\bbH_d$  is the set $\bbC^d\times\bbR$ endowed with product 
$$
[w,s][z,t]  =  \big[ w+z, s+t -\textstyle{\frac12} \Im (w\cdot\bar z)\big] \,.
$$}
\end{definition}
The right and left Haar
measures on the Heisenberg group coincide with the Lebesgue measure on $\bbC^d\times\bbR$. In particular, the Lebesgue measure is both left and right translation invariant. 

\medskip

We now recall the basic facts for the Fourier transform on the Heisenberg
group. For $\lambda\in\bbR\setminus\{0\}$ define the Fock space
\begin{equation*}
\cF^\lambda 
=\bigg\{ F \in\Hol(\bbC^d):\ 
\bigg( \frac{|\lambda|}{2\pi} \bigg)^d \int_{\bbC^d} |F(z)|^2\,
e^{-\frac\lambda2 |z|^2}dz <+\infty \bigg\}
\end{equation*}
when $\lambda>0$,   and $\cF^\lambda = \cF^{|\lambda|}$ when  
$\lambda<0$. The Fock space is a reproducing kernel Hilbert space with reproducing kernel $e^{\frac{|\lambda|}{2}z\cdot \ov w}$. A complete orthonormal basis of $\cF^\lambda$ is given by the normalized monomials $\{z^\alpha/\|z^\alpha\|_{\cF^\lambda}\}_{\alpha\in\bbN_0^d}$, where
$$
\|z^\alpha\|_{\cF^\lambda}^2 
= \alpha!
\bigg(\frac{2}{|\lambda|}\bigg)^{|\alpha|}
\,.
$$

For $[z,t]\in\bbH_d$, the
Bargmann 
representation $\sigma_\lambda[z,t]$ is the operator acting on
$\cF^\lambda$ given by, 
\begin{equation*}
\sigma_\lambda[z,t] F (w)
= e^{i\lambda t-\frac\lambda2 w\cdot\ov z -\frac\lambda4 |z|^2}
F(w+z) 
\end{equation*}
if $\lambda>0$, and, if $\lambda<0$, as 
$\sigma_\lambda[z,t] = \sigma_{-\lambda} [\ov z,-t]$, that is, 
\begin{equation*}
\sigma_\lambda[z,t] F (w)
= e^{i\lambda t+\frac\lambda2 w\cdot z +\frac\lambda4 |z|^2}
F(w+\bar z) \,.
\end{equation*}

If $f\in L^1(\bbH_d)$, for $\lambda\in\bbR\setminus\{0\}$, 
$\sigma_\lambda(f)$ is the operator 
acting on $\cF^\lambda$ as
$$
 \sigma_\lambda (f) F (w) 
= \int_{\bbH_d} f[z,t] \sigma_\lambda[z,t]F (w)\, dzdt\,. 
$$
If $f\in L^2(\bbH_d)$, we have Plancherel's formula
\begin{equation*}
\| f\|_{L^2(\bbH_d)}^2 = \frac{1}{(2\pi)^{d+1} }
\int_{\bbR}  \|\sigma_\lambda(f)\|_{\HS}^2 |\lambda|^d\,
d\lambda\,,
\end{equation*}
where $\|\sigma_\lambda(f)\|^2_{\HS}=\sum_{\alpha}\|\sigma_{\lambda}(f)e_\alpha\|^2_{\cF^\lambda}$ is the Hilbert--Schmidt norm of $\sigma_\lambda(f)$ and $e_\alpha=z^\alpha/\|z^\alpha\|_{\cF^\lambda}$. 
If $f\in L^1\cap L^2(\bbH_d)$ the following inversion formula holds:
\begin{equation*}
f[z,t] = \frac{1}{(2\pi)^{d+1}}  
\int_{\bbR}  \tr \big( \sigma_\lambda(f)\sigma_\lambda[z,t]^*\big)
|\lambda|^d\, d\lambda .
\end{equation*}

\subsection{The Drury--Arveson space on $\mathcal{U}$} In \cite{AMPS} a family of
holomorphic function spaces on $\cU$ depending on a real parameter
$\nu$ was studied and characterized by means of the group Fourier
transform on $\bbH_d$. This family of spaces includes weighted Bergman
spaces, the Hardy space, weighted Dirichlet spaces and the Dirichlet
space. In particular, the Drury--Arveson space $DA$ was
identified as a particular weighted Dirichlet space. Here we recall the results in \cite{AMPS} we need in the rest of the paper.

\begin{definition}\label{L2-DA}{\rm
We define  the space 
$\cL^2_{DA} $ as the space of functions $\tau$
 on $\bbR\setminus\{0\}$ such that:
\begin{align*}
 \begin{split}
\rm{(i)}&\textrm{ $\tau(\lambda) \in \HS(\cF^\lambda)$  for
  every $\lambda$, i.e., $\tau(\lambda):\cF^\lambda\to\cF^\lambda$ is a Hilbert--Schmidt operator;}\\
\rm{(ii)}&\textrm{ $\tau(\lambda)=0$ for
  $\lambda>0$;}\\
 \rm{(iii)}& \textrm{ $\ran( \tau(\lambda)) \subseteq
   \operatorname{span}\{1\}$;}\\
\rm{(iv)}& \textrm{ $
\displaystyle{ \|\tau\|_{\cL^2_{DA}}^2: = \frac{1}{(2\pi)^{d+1}} 
\int_{-\infty}^0 \|\tau(\lambda)\|_{\HS}^2 \,
|\lambda|^{2d} d\lambda <+\infty }$.} 
 \end{split}
\end{align*}}
\end{definition}

The following result holds true. 
\begin{theorem}[\cite{AMPS}]\label{DA-AMPS}
Let $n>\frac{d}{2}$. 
Let $f\in DA_{(n)}$  defined as in \eqref{DA-m-def}.  Then, there exists $\tau\in \cL^2_{DA} $ such
that, for $(\z,\z_{d+1})\in\cU$,
\begin{equation}\label{PW-Dnu-eq1}
f(\z,\z_{d+1}) = (f\circ\Psi^{-1})(z,t,h) =
\frac{1}{(2\pi)^{d+1}} \int_{-\infty}^0 
 e^{h\lambda} \tr \big(
\tau(\lambda)\sigma_\lambda[z,t]^*\big) \, |\lambda|^d d\lambda\,,
\end{equation}
and
\begin{equation}\label{PW-Dnu-eq2}
\| f\|_{DA_{(n)}}^2
= \frac{\Gamma(2n-d)}{2^{2n-d}} 
\| \tau\|_{\cL^2_{DA}}^2 \,. 
\end{equation}

Conversely, given $\tau\in\cL^2_{DA}$, let $f$ be defined as in
\eqref{PW-Dnu-eq1}. Then $f\in DA_{(n)}$ and \eqref{PW-Dnu-eq2}
holds. 
Therefore, for each $n>\frac{d}{2}$, the spaces $DA_{(n)}$
all coincide and their norms satisfy \eqref{PW-Dnu-eq2}.  
\end{theorem}

Hence, we simply write $DA$ in place of $DA_{(n)}$.

\subsection{Semigroups of operators}

A semigroup of operators on a Hilbert space $\cH$ is a one parameter
family of bounded linear operators $\{T_t\}_{t\geq 0}$ on $\cH$ such
that:
\begin{enumerate}
    \item[(i)]$T_0 = \Id $;
    \item[(ii)] $T_{t+s}=T_t T_s, \,\,\, \forall s ,t > 0. $
\end{enumerate}
If in addition $ T_t $ converges to the identity operator $\Id$ in the strong operator topology as $t \searrow 0$, the semigroup is called {\it strongly continuous} or $C_0$. From now on we will focus exclusively on $C_0$ semigroups.
 The infinitesimal generator $G$ of a $C_0$ semigroup $\{T_t\}_{t\geq
   0}$  is a linear operator defined on the subspace $\Dom(G)$ of
 $v\in \cH$ such that the limit  
$$
\lim_{t\searrow 0 } \frac{T_t v - v}{t}=: G v 
$$
exists in the norm topology. It can be shown that $G$ on its domain
$\Dom(G)$ is a linear densely defined closed operator \cite[Section
34, Theorem 4]{lax2002functional}. In particular, we will be interested
in the characterization of infinitesimal generators for contraction
semigroups, i.e., semigroups $\{T_t\}_{t\geq 0}$ such that each $T_t$
is a contraction, which is provided by the Lumer-Philips
Theorem.

\begin{definition}
A densely define operator $G$ on a Hilbert space $\cH$ is called {\it dissipative} if 
$$ \Re\, \inner{Gv}{v}_{\cH} \leq 0, \,\,\, \forall v \in \Dom(G). 
$$
It is called {\it maximal dissipative},  if it is dissipative and its resolvent set  $r(G)$ includes $\bbR_+=(0,+\infty)$.
\end{definition}

For the following renowned theorem we refer the reader, for instance, to \cite[p. 432]{lax2002functional}.

\begin{theorem}[Lumer-Philips]\label{Lumer-Philips} A densely defined operator $G$ is the
  infinitesimal generator of a (unique) semigroup of contractions if
  and only if it is maximal dissipative. 
\end{theorem}

The following is a very well-known lemma that we shall need in what follows.

\begin{lemma}\label{semigroup:lemma}
Let $\{T_t\}_{t\geq0}$ be a semigroup of bounded operators on a Hilbert space $\cH$. Suppose that:
\begin{enumerate}
\item[\rm (i)]  there exists $ \delta > 0 $ such that $\sup_{0<t<\delta} \norm T_t\norm _{\cB(\cH)} < +\infty$;
\item[\rm (ii)] for some dense subset $\cD \subseteq \cH $, $ T_t f \rightarrow f$ for all $f \in \cD$.  
\end{enumerate}
\noindent Then, $\{T_t\}_{t\geq0}$ is a strongly continuous semigroup. 
\end{lemma}

We remark that we will work with semigroups that appear to have a negative parameter $\tau$; this is to stay consistent with \cite{AMPS}. When we use this notation we simply mean that the semigroup has the positive parameter $t=-\tau$.
%
%
%

\section{Proof of Theorem \ref{isometry-Phi}}\label{DA-section}

The proof of Theorem \ref{isometry-Phi} follows at once from the next two lemmas.
Recall that the measure $\mu$ on  $\bbN_0^d\times\bbR_-$ is defined in \eqref{measure-mu}.
Now define
\begin{equation*}
\varphi(\alpha,\lambda)=\|z^\alpha\|^{-1}_{\cF^\lambda}\langle e_0,
\sigma_{\lambda}(f_0)(e_\alpha)\rangle_{\cF^{\lambda}}. 
\end{equation*}

\begin{lemma}\label{isometry-lem}
 The map  $\Phi: DA\cap H^2\to L^2(\bbN_0^d\times\bbR_-, d\mu)$ defined as
 \begin{equation}\label{map-Phi}
 \Phi(f) (\alpha,\lambda)= \varphi(\alpha,\lambda),
 \end{equation}
 extends to
a conjugate linear isometry $\Phi: DA\to
L^2(\bbN_0^d\times\bbR_-, d\mu)$.  Its inverse is the map $\sS$ of
Theorem \ref{isometry-Phi}.
\end{lemma}

\begin{proof}
 First we assume that $f\in DA\cap H^2$ where $H^2$ denotes the Hardy
 space on $\cU$. This intersection is dense in $DA$ ( \cite[Lemma
 4.2]{AMPS}) and every $f\in H^2$ admits a boundary value function
 $f_0\in L^2(\p\cU)$. Moreover, for such function $f$ the function
 $\tau\in\cL^2_{DA}$ in formula \eqref{PW-Dnu-eq1} actually coincides
 with the Fourier transform of its boundary value function, that is,
 $\tau(\lambda)=\sigma_{\lambda}(f_0)$ (see \cite{AMPS, OV}). 

Let $\{e_\alpha\}_\alpha$ be the orthonormal basis of normalized monomials of the Fock space $\cF^{\lambda}$ and let $F=\sum_{\alpha}F_\alpha e_\alpha$ be a function in $\cF^{\lambda}$. Then, for every function $f\in DA\cap H^2$, we have
\begin{align*}
    \sigma_\lambda(f_0)(F)&=\sum_{\alpha\in \bbN^d_0}F_\alpha\sigma_{\lambda}(f_0)(e_\alpha)=\sum_{\alpha \in \bbN^d_0}F_\alpha\langle
\sigma_{\lambda}(f_0)(e_\alpha),
e_0\rangle_{\cF^\lambda}e_0 
\end{align*}
thanks to property (iii) in Definition \ref{L2-DA}. Thus,
\begin{align*}
    \sigma_{\lambda}(f_0)(F)=\langle F, \Phi_\lambda\rangle_{\cF^\lambda}e_0
\end{align*}
where 
\begin{align*}
\Phi_\lambda&=\sum_{\alpha \in \bbN^d_0}\ov{\langle  \sigma_{\lambda}(f_0)(e_\alpha), e_0\rangle_{\cF^\lambda}}e_\alpha\\
&=\sum_{\alpha \in \bbN^d_0}\langle e_0, \sigma_{\lambda}(f_0)(e_\alpha)\rangle_{\cF^{\lambda}}e_\alpha\\
&=\sum_{\alpha \in \bbN^d_0} \|z^\alpha\|^{-1}_{\cF^\lambda}\langle e_0, \sigma_{\lambda}(f_0)(e_\alpha)\rangle_{\cF^{\lambda}} z^\alpha.
\end{align*}
In particular, 
we deduce that
$\|z^\alpha\|_{\cF^{\lambda}}\varphi(\alpha,\lambda)=\langle
\Phi_\lambda,e_\alpha\rangle_{\cF^\lambda}$. 
Now,
\begin{align*}
 \|\sigma_{\lambda}(f_0)\|^2_{\HS}&=\sum_{\alpha\in\NN_0^d}\|\sigma_{\lambda}(f_0)e_\alpha\|_{\cF^{\lambda}}^2\\
 &=\sum_{\alpha\in\NN_0^d}|\langle e_\alpha, \Phi_\lambda\rangle_{\cF^{\lambda}}|^2\|e_0\|_{\cF^{\lambda}}^2\\
 &=\sum_{\alpha\in\NN_0^d}|\langle e_\alpha, \Phi_\lambda\rangle_{\cF^{\lambda}}|^2,
\end{align*}
since $\{e_\alpha\}_\alpha$ is an orthonormal basis in $\mathcal F^\lambda$. However,
$$
|\langle e_\alpha, \Phi_\lambda\rangle_{\cF^{\lambda}}|=\|z^\alpha\|_{\cF^{\lambda}}|\varphi(\alpha,\lambda)|,
$$
so that
$$
 \|\tau(\lambda)\|^2_{\HS}=\sum_{\alpha\in\NN_0^d}\alpha! (2/|\lambda|)^{|\alpha|}|\varphi(\alpha,\lambda)|^2.
$$ 

In conclusion,
\begin{align}\label{isometry}
\begin{split}
\|f\|_{DA}^2&= c_d \int_{-\infty}^0\sum_{\alpha\in\NN_0^d}\alpha! (2/|\lambda|)^{|\alpha|}|\varphi(\alpha,\lambda)|^2 |\lambda|^{2d}\, d\lambda\\
&=c_d\int_{\NN_0^d\times\mathbb R_{-}}|\varphi(\alpha,\lambda)|^2\, d\mu
(\alpha,\lambda).  
\end{split}
\end{align}
\end{proof}

We now see that the isometry $\Phi$ is surjective as well. We first specialize the inversion formula \eqref{PW-Dnu-eq1}.

\begin{lemma}\label{inversione-DA-lem}
 Let $f\in DA$. Then, with the notation above and setting $\z=(\z_1,\ldots,\z_d)$, we have
\begin{equation}\label{inversion-DA}
f(\z,\z_{d+1})= \frac{1}{(2\pi)^{d+1}} \int_{\bbN_0^d\times\bbR_{-}}  \z^\alpha e^{-i\lambda\z_{d+1}} \overline{\varphi(\alpha,\lambda)}\,|\lambda|^d\, d\alpha d\lambda.
\end{equation}
\end{lemma}
\begin{proof}
 If $(\z,\z_{d+1})=\Psi^{-1}(z,t,h)$ we know from Theorem \ref{DA-AMPS} that there exists $\tau\in\cL^2_{DA}$ such that
 \begin{equation*}
 (f\circ\Psi^{-1})(z,t,h)= \frac{1}{(2\pi)^{d+1}} \int_{-\infty}^0 
 e^{h\lambda} \tr \big(
\tau(\lambda)\sigma_\lambda[z,t]^*\big) \, |\lambda|^d d\lambda\,.
 \end{equation*}
 Using the fact that $\tau(\lambda)$ is a rank one operator such that $\rank\{\tau(\lambda)\}\subseteq\Span\{e_0\}$ for every $\lambda$  we get
\begin{align*}
  \tr(\tau(\lambda)\sigma_\lambda[z,t]^*) &=  \tr(\sigma_\lambda[z,t]^*\tau(\lambda))\\
  &= \sum_{\alpha\in\NN_0^d}\inner{\tau(\lambda)e_\alpha}{\sigma_\lambda[z,t]e_\alpha}\\
&=\sum_{\alpha\in\NN_0^d}\inner{\tau(\lambda)e_\alpha}{P_0\sigma_\lambda[z,t]e_\alpha}
\end{align*}
where $P_0$ denotes the orthogonal projection onto the subspace generated by $e_0$. Moreover, it holds (see also $(26)$ in \cite{AMPS})
$$
P_0 \sigma_\lambda[z,t] e_\alpha=\bigg( \frac{1}{\sqrt{\alpha!}}
\bigg(\frac{|\lambda|}{2}\bigg)^{|\alpha|/2} 
\bigg( \frac{|\lambda|}{2\pi} \bigg)^n 
e^{i\lambda t+ \frac\lambda4
  |z|^2} 
\bar z^\alpha \bigg) e_0.
$$
Thus, 
\begin{align*}
 \tr(\tau(\lambda)\sigma_\lambda[z,t]^*) &=\sum_{\alpha\in\NN_0^d}\inner{\tau(\lambda)e_\alpha}{P_0\sigma_\lambda[z,t]e_\alpha}\\
  & = \sum_{\alpha\in\NN_0^d} \inner{e_\alpha}{\Phi_\lambda}\inner{e_0}{P_0\ssl[z,t]e_\alpha} \\
   & = \sum_{\alpha\in\NN_0^d} \inner{e_\alpha}{\Phi_\lambda}\overline{\inner{P_0\ssl[z,t]e_\alpha}{e_0}}\\ 
   & = \sum_{\alpha\in\NN_0^d}  \frac{1}{\sqrt{\alpha!}} \Big( \frac{|\lambda|}{2} \Big)^{|\alpha|/2}e^{-i\lambda t+\frac{\lambda}{4}|z|^2}z^\alpha  \inner{e_\alpha}{\Phi_\lambda} \\
  & = \sum_{\alpha\in\NN_0^d}   \frac{1}{\sqrt{\alpha!}} \Big( \frac{|\lambda|}{2} \Big)^{|\alpha|/2} \|z^\alpha\|\overline{\varphi(\alpha,\lambda)} e^{-i\lambda t+\frac{\lambda}{4}|z|^2}z^\alpha.\\
  &=\sum_{\alpha\in\bbN_0^d}\overline{\varphi(\alpha,\lambda)} e^{-i\lambda t+\frac{\lambda}{4}|z|^2}z^\alpha.
\end{align*}
Hence,
 \begin{align*}
  f(\zeta, \zeta_{d+1}) = (f\circ\Psi^{-1})(z,t,h)&=\frac{1}{(2\pi)^{d+1}} \int_{-\infty}^0 
 e^{h\lambda} \sum_{\alpha\in\NN_0^d}   \overline{\varphi(\alpha,\lambda)} e^{-i\lambda t+\frac{\lambda}{4}|z|^2}z^\alpha \, |\lambda|^d d\lambda\,\\
 &= \frac{1}{(2\pi)^{d+1}} \int_{\bbN_0^d\times\bbR_{-}}  z^\alpha e^{-i\lambda(t-i(\frac{|z|^2}{4}+h))} \overline{\varphi(\alpha,\lambda)}\,|\lambda|^d\, d\alpha d\lambda\\
 &=  \frac{1}{(2\pi)^{d+1}} \int_{\bbN_0^d\times\bbR_{-}}  \z^\alpha e^{-i\lambda\z_{d+1}} \overline{\varphi(\alpha,\lambda)}\,|\lambda|^d\, d\alpha d\lambda
 \end{align*}
as we wished to show.
 \end{proof}

\section{Pointwise multipliers on the Drury--Arveson space} \label{multipliers-section}
In this section we explicitly study some multiplier operators on $DA$ and we prove Theorem \ref{shift-adjoint}. Recall that given
a function $m$ the associated multiplier operator is the operator
$f\mapsto mf$. The problem of characterizing the multiplier algebra of
a given reproducing kernel Hilbert function space is a classical
problem. In the case of the Drury--Arveson space on the unit ball
$DA(\mathbb B^{d+1})$ this problem turned out to be very challenging;
a first important result is due to J. Ortega and J. Fàbrega 
\cite{ortega_fabrega}.  Their result reads as follows. Let $n\in\bbN_0$ be such that $2n>d$ and for $f$ in $\operatorname{Hol}(\mathbb B^{d+1})$ define the measure 
$$
d\nu_{f}(w)=|\mathcal R^n f(w)|^2(1-|w|^2)^{2n-d-1}\, d\nu(w),
$$
where $\mathcal R$ is the radial derivative, $d\nu$ is the normalized
Lebesgue measure on $\mathbb B^{d+1}$ and $w\in\mathbb B^{d+1}$. Then,
$f$ is a $DA(\mathbb B^{d+1})$-multiplier if and only
if $f\in H^\infty(\mathbb B^{d+1})$ and $d\nu_f$ is a Carleson measure for $DA(\mathbb B^{d+1})$. A few years later
N. Arcozzi, R. Rochberg and E. Sawyer in
\cite{ARS_carleson} completely characterized  the Carleson measures of $DA(\mathbb B^{d+1})$; see also \cite{tch1,tch2}, \cite{VW}.  Hence, the
multiplier algebra of the Drury--Arveson space on the unit ball is
completely characterized. Nonetheless, there is still interest in
finding an easier characterization and we refer the reader, for
instance, to \cite{fang_xia}. 

From \cite{ortega_fabrega, ARS_carleson} we can also deduce an
indirect characterization of the multiplier algebra for $DA$ on the
Siegel half-space. Indeed, let $\mathcal C$ be the multi-dimensional
Cayley transform defined in \eqref{cayley_transform}. Then, up to
an irrelevant multiplicative constant,  
$$
K^\cU_{DA}\big(\mathcal C(z,z_{d+1}),\mathcal C(w,w_{d+1})\big)
=(1-z_{d+1})K^{\mathbb B^{d+1}}_{DA}\big((z,z_{d+1}),(w,w_{d+1})\big)(1-\overline{w_{d+1}})
$$
where $K^\cU_{DA}$ and $K^{\mathbb B^{d+1}}_{DA}$ denote the
reproducing kernel of $DA$ and $DA(\mathbb B^{d+1})$
respectively, and $(z,z_{d+1}),(w,w_{d+1})\in\bbB^{d+1}$.
From the abstract theory of reproducing kernel Hilbert spaces (see, for instance, \cite[Chapter 2.6]{agler_mccarthy}) we deduce that 
$$
f\mapsto (1-z_{d+1})^{-1}(f\circ\mathcal C)
$$
is a surjective isometry from $DA$ onto $DA(\bbB^{d+1})$ and that $m$ is a multiplier for $DA$ if and only if $(m\circ\mathcal C)$ is a multiplier for $DA(\bbB^{d+1})$. 

For our goal we need to study the multiplier operators associated to
the functions $\z^\gamma=\z_1^{\gamma_1}\cdots\z_{d}^{\gamma_d}$,
$e^{-i\tau \z_{d+1}}$ for $\tau<0$ and  $\z^\gamma e^{-i\tau
  \z_{d+1}}$. However, we do not rely on the multiplier
characterization on the unit ball since it is easier to study them
directly. The proof of the following lemma is standard and we omit it.
\begin{lemma}\label{operator-domain}
The set
 $$
 \cD=\big\{f\in L^2(\bbN_0^d\times\bbR_{-}, d\mu):  \supp{f}\textrm{ is compact and }f(\alpha,\cdot)\in C_c^{\infty}\, \forall\alpha\in\bbN_0^d\big\}.
 $$
is dense in $L^2(\bbN_0^d\times\bbR_{-}, d\mu)$.
 \end{lemma}

We now study the multiplier operator associated to the monomial $\z^\gamma, \gamma\in\bbN_0^d$. Although this operator is unbounded on $DA$, as it is easily seen, it is closed and densely defined.
\begin{lemma}\label{shift-densely-defined}
 Let  $ \gamma \in \bbN_0^d $. Then, the multiplier operator $m_\gamma$ associated to the function $\z^\gamma=\z_1^{\gamma_1}\cdots \z_d^{\gamma_d}$ is a closed densely defined operator on $DA$.
\end{lemma}
\begin{proof}
Let us consider $\cD_\gamma=\{f\in DA: \z^\gamma f\in DA\}$ as domain of our multiplier operator.  Let $\{f_n\}_n\subseteq \cD_\gamma$ be a sequence such that $f_n\to f$ in $DA$ and $\z^\gamma f_n\to g\in DA$. Since $DA$ is a reproducing kernel Hilbert space we also have that $f_n\to f$ and $\z^\gamma f_n\to g$ uniformly on compact sets. Hence, $\z^\gamma f_n\to \z^\gamma f$ uniformly on compact sets as well and $\z^\gamma f=g$. In particular both $f$ and $\z^\gamma f$ are in $DA$ so that $f$ belongs to the domain $\cD_\gamma$. Thus, our operator is closed. 

To prove that $\cD_\gamma$ is dense we exploit the previous Lemma
\ref{operator-domain}. Let $\varphi\in \cD$ and define $f\in DA$ with
the inversion formula \eqref{inversion-DA}. Then, 
\begin{align*}
  \z^\gamma f(\z,\z_{d+1})
  & = \frac{1}{(2\pi i)^{d+1}}\int_{\bbN_0^d\times\bbR_-}
    \z^{\alpha+\gamma} e^{-i\lambda \z_{d+1}}
    \overline{\varphi(\alpha,\lambda)}  \,|\lambda|^d\, d\alpha d\lambda\\
& = \frac{1}{(2\pi)^{d+1}}\int_{\supp\varphi}\z^{\alpha} e^{-i
  \lambda \z_{d+1}}  \overline{\varphi(\alpha-\gamma ,\lambda)} 
 \,|\lambda|^d \chi_{\{\alpha\geq \gamma \}}(\alpha)\, d\alpha d\lambda
\end{align*}
where $\chi_{\{\alpha\geq \gamma \}}(\alpha)$ is the characteristic function of the set $\{\alpha\in\bbN_0^d: \alpha_i \geq  \gamma_i, i=1,\ldots,d \}$. Recall that $\supp\varphi$ is a compact subset of $\bbN_0^d\times\bbR_-$ and assume that $(\alpha,\lambda)\in\supp\varphi$ implies $|\alpha|<N$ for some positive integer $N$ and $\lambda\in I$ where $I$ is a compact subset of $\bbR_-$. From \eqref{isometry} we have
\begin{align*}
 \| \z^\gamma f\|^2_{DA} & = c_d \int_I\sum_{|\alpha-\gamma|<N}|\varphi(\alpha-\gamma,\lambda)|^2\, \alpha! \big(\frac{2}{|\lambda|}\big)^{|\alpha|}|\lambda|^{2d} \chi_{\{\alpha\geq \gamma\}}(\alpha)\, d\alpha d\lambda\\
                        & = c_d\int_I\sum_{|\alpha|<N} |\varphi(\alpha,\lambda)|^2 \frac{(\alpha+\gamma)!}{\alpha !}\Big(\frac{2}{|\lambda|}\Big)^{|\gamma|}\, d\mu(\alpha,\lambda)\\
                        & \leq c_d(\varphi)\int_I\sum_{|\alpha|<N} |\varphi(\alpha,\lambda)|^2 \, d\mu(\alpha,\lambda)\\
                        & = c_d(\varphi)\|f\|_{DA}^2
\end{align*}
 where $c_d(\varphi)$ is a constant depending on the compact support of $\varphi$ (and on the dimension $d$). The density of $\cD_\gamma$ in $DA$ now follows from the density of $\cD$ in $L^2(\bbN_0^d\times\bbR_-,\ d\mu)$ and Theorem \ref{isometry-Phi}.
\end{proof}

We now investigate the other multiplier operators we are interested in and see that they are actually bounded on the Drury--Arveson space.

\begin{lemma}\label{exp-multiplier}
 The multiplier operators associated to the functions $\{ e^{-i\tau \z_{d+1}} \}_{\tau<0}$ form a strongly continuous semigroup of contractions on $DA$ with corresponding infinitesimal generator the multiplier operator corresponding to the function  $i \z_{d+1}$.
\end{lemma}
\begin{proof}
From the inversion formula \eqref{inversion-DA} we have 
 \begin{align*}
  e^{-i\tau \z_{d+1}}&f(\z,\z_{d+1})= \frac{1}{(2\pi i)^{d+1}}\int_{\bbN_0^d\times\bbR_-} \z^\alpha e^{-i
  (\lambda+\tau)\z_{d+1}} \overline{\varphi(\alpha,\lambda)}\,|\lambda|^d\, d\alpha d\lambda\\
  &=\frac{1}{(2\pi i)^{d+1}}\int_{\bbN_0^d\times\bbR_-} \z^\alpha e^{-i
  \lambda \z_{d+1}} \overline{\varphi(\alpha,\lambda-\tau)}\big|\frac{\lambda- \tau}\lambda\big|^d\chi_{(-\infty,\tau)}(\lambda)\,|\lambda|^d\, d\alpha d\lambda
 \end{align*}
where $\chi_{(-\infty,\tau)}$ is the characteristic function of the interval $(-\infty,\tau)$. Thus, by \eqref{isometry} we get
\begin{align*}
 \| e^{-i\tau \z_{d+1}}f\|_{DA}^2&=\int_{\bbN_0^d\times\bbR_-}|\varphi(\alpha,\lambda-\tau)|^2 \big|\frac{\lambda-\tau}\lambda\big|^{2d}\chi_{(-\infty,\tau)}(\lambda)\, d\mu(\alpha,\lambda)\\
 &=\int_{\bbN_0^d\times\bbR_-} |\varphi(\alpha,\lambda)|^2\big|\frac{\lambda}{\lambda+\tau}\big|^{|\alpha|}\, d\mu(\alpha,\lambda)\\
 &\leq \int_{\bbN_0^d\times\bbR_-} |\varphi(\alpha,\lambda)|^2\, d\mu(\alpha,\lambda)\\
 &=\|f\|_{DA}^2.
\end{align*}
Hence, we conclude that $ e^{-i\tau \z_{d+1}}$ is a contractive multiplier. The semigroup property is automatically satisfied. Since $ e^{-i\tau \z_{d+1}}f $ converges in norm to $f$ for all $f\in \cD$ and the semigroup has uniformly bounded seminorm, it is strongly continuous by Lemma \ref{semigroup:lemma}. Then, by definition, the infinitesimal generator is given by
\begin{align*}
    \lim_{\tau\to 0^-}\frac{(e^{-i\tau \z_{d+1}}-1)f}{-\tau}=i\z_{d+1}f
\end{align*}
and the proof is concluded.
\end{proof}


\begin{lemma} \label{boundedmult}
Let $(\gamma,\tau)\in\bbN_0^d\times\bbR_-, \tau < 0$. Then, the multiplier operator associated to the muliplier $\z^{\gamma} e^{-i\tau \z_{d+1}}$ extends to a bounded operator on $DA$.
\end{lemma}
\begin{proof}
Let $\cD$ be as in Lemma \ref{operator-domain} and let $f\in \cD $. We have
 \begin{align*}
  &\z^\gamma e^{-i\tau \z_{d+1}} f(\z,\z_{d+1}) =\frac{1}{(2\pi i)^{d+1}}\int_{\bbN_0^d\times\bbR_-} \z^{\alpha+\gamma} e^{-i
  (\lambda+\tau)\z_{d+1}} \overline{\varphi(\alpha,\lambda)}\,|\lambda|^d\, d\alpha d\lambda\\
  &=\frac{1}{(2\pi i)^{d+1}}\int_{\bbN_0^d\times\bbR_-} \!\!\!\!\!\!\!\!\z^{\alpha} e^{-i
  \lambda \z_{d+1}} \overline{\varphi(\alpha-\gamma,\lambda-\tau)}\,|\lambda-\tau|^d \chi_{(-\infty,\tau)}(\lambda)\chi_{\{\alpha\geq\gamma\}}(\alpha)\, d\alpha d\lambda\\
  &=\frac{1}{(2\pi i)^{d+1}}\int_{\bbN_0^d\times\bbR_-} \!\!\!\!\!\!\!\!\z^{\alpha} e^{-i
  \lambda \z_{d+1}} \overline{\varphi(\alpha-\gamma,\lambda-\tau)}\,\big|\frac{\lambda-\tau}\lambda\big|^d\chi_{(-\infty,\tau)}(\lambda)\chi_{\{\alpha\geq\gamma\}}(\alpha) |\lambda|^d d\alpha d\lambda
 \end{align*}
where $\chi_{\{\alpha\geq\gamma\}}(\alpha)$ is the characteristic function of the set $\{\alpha\in\bbN_0^d: \alpha_i\geq\gamma_i, i=1,\ldots,d\}$. Hence, setting $m_{\gamma,\tau}(\z,\z_{d+1})=\z^\gamma e^{-i\tau \z_{d+1} }$,
\begin{align*}
\begin{split}
 \|m_{\gamma,\tau}f\|_{DA}^2&=\int_{\bbN_0^d\times\bbR_-}|\varphi(\alpha-\gamma,\lambda-\tau)|^2\,\big|1-\frac\tau\lambda\big|^{2d}\chi_{(-\infty,\tau)}(\lambda)\chi_{\{\alpha\geq\gamma\}}(\alpha)\, d\mu(\alpha,\lambda)\\
 &=\int_{\bbN_0^d\times\bbR_-}|\varphi(\alpha,\lambda)|^2 2^{|\gamma|}\frac{|\lambda|^{|\alpha|}}{|\lambda+\tau|^{|\alpha+\gamma|}}\frac{(\alpha+\gamma)!}{\alpha!}\, d\mu(\alpha,\lambda).
\end{split}
\end{align*}

Therefore, 
\[  \|m_{\gamma,\tau} \|_{\cM(DA)}^2 \leq \sup_{\alpha \in \bbN_0^d, \lambda < 0 } 2^{|\gamma|}\frac{|\lambda|^{|\alpha|}}{|\lambda+\tau|^{|\alpha+\gamma|}}\frac{(\alpha+\gamma)!}{\alpha!}. \]

In order to calculate this supremum we first maximize over $ \lambda $. For symmetry we can consider $\lambda>0$ and $\tau>0$. We find that when $|\alpha|\neq 0$ the maximum is achieved for $\lambda = \frac{|\alpha|}{|\gamma|}\tau,$ whereas for $\alpha = 0 $ the expression is decreasing in $\lambda$ and therefore the maximum is achieved at $\lambda = 0$.

Substituting the value for $\lambda $ in the expression above we find 
\begin{align*}
	2^{|\gamma|}& \frac{|\alpha|^{|\alpha|}\tau^{|\alpha|}}{|\gamma|^{|\alpha|}} \frac{1}{\tau^{|\alpha+\gamma|}\big(\frac{|\alpha|}{|\gamma|}+1\big)^{|\alpha+\gamma|}}\frac{(\alpha+\gamma)!}{\alpha!}  =\\
	&=\big( \frac{2}{\tau }\big)^{|\gamma|} \Big( 1 + \frac{|\gamma|}{|\alpha|}  \Big)^{-|\alpha|} \frac{(\alpha + \gamma ) !}{\alpha ! } \Big( \frac{|\alpha|}{|\gamma|} +1  \Big)^{-|\gamma|} .
\end{align*}

Note that the first factor is constant, the second is decreasing in $|\alpha|$ and tends to $e^{-|\gamma|}$ as $|\alpha| \to + \infty .$   For the last term notice that 
\begin{align*}
 \frac{(\alpha + \gamma )! }{\alpha!} \Big( \frac{|\alpha|}{|\gamma|} + 1\Big)^{-|\gamma|} & = \prod_{i=1}^d \frac{(\alpha_i + \gamma_i)!}{\alpha_i ! \Big( \frac{|\alpha|}{|\gamma|}+1 \Big)^{\gamma_i} } \\
 & \leq   \prod_{i=1}^d \frac{(|\alpha| + \gamma_i)!}{|\alpha| ! \Big( \frac{|\alpha|}{|\gamma|}+1 \Big)^{\gamma_i} }  \\
 & \leq \gamma! \prod_{i=1}^d \Big( \frac{|\alpha| + 1}{\frac{|\alpha|}{|\gamma|}+1} \Big)^{\gamma_i} \\
 	&  \leq \gamma! |\gamma|^{|\gamma|}.
\end{align*}

This proves that the multiplier operator associated to $m_{\gamma,\alpha}$ is bounded with norm less than
\[ \max \Big\{ \gamma !  \big( \frac{2}{\tau}\big)^{|\gamma|}, \gamma!   \big( \frac{2}{\tau}\big)^{|\gamma|} \frac{|\gamma|!}{\sqrt{2\pi |\gamma|}} \Big\}, \] 
where we used Stirling's asymptotic.
\end{proof}

We now prove Theorem \ref{shift-adjoint}

\begin{proof}
The proof of the unitary equivalence of the two operators follows essentially from the computation in Lemma \ref{boundedmult}. 
For the computation of the adjoint operator, if $\langle\cdot,\cdot\rangle$ denotes the inner product in $L^2(\bbN_0^d\times\bbR_-, d\mu)$, we have
\begin{align*}
    \langle S^*_{\gamma,\tau}\varphi,\psi\rangle&= \langle \varphi, S_{\gamma,\tau}\psi\rangle\\
    &=\int_{\{\alpha\geq\gamma\}}\int_{\{\lambda<\tau\}}\varphi(\alpha,\lambda) \frac{|\lambda-\tau|^d}{|\lambda|^d}\overline{\psi(\alpha - \gamma, \lambda - \tau)}\, d\mu(\alpha,\lambda)\\
    &=\int_{\bbN_0^d\times\bbR_-} \frac{|\lambda+\tau|^{d-|\alpha|-|\gamma|}}{|\lambda|^{d-|\alpha|}} \frac{(\alpha+\gamma)!}{\alpha ! } 2^{|\gamma|} \varphi(\alpha+\gamma, \lambda + \tau ) \overline{\psi(\alpha,\lambda)}\, d\mu(\alpha,\lambda)
\end{align*}
and the conclusion follows.
\end{proof}

\section{The von Neumann type inequality}\label{lifting-section}

In this section we prove our main results Theorem \ref{lifting-thm} and Theorem \ref{thm-vonNeumann}. We first prove a couple of lemmas.

\begin{lemma}\label{commutation}
Let $T$ be a bounded operator on a Hilbert space $\cH$ and let $U$ be a densely defined closed operator. Assume that $U$ is the infinitesimal generator of a (unique) $C_0$ semigroup  $\{e^{\tau U}\}_{\tau>0}$ and assume that $T$ and $U$ strongly commute. Then, $T$ and $e^{\tau U}$ commute for all $\tau >0 $.
\end{lemma}
\begin{proof}
Consider some $\lambda_0 \in r(T)$, the resolvent set of the operator $T$. Then an algebraic computation shows that the bounded operators 
\[  Z_\tau : = (T-\lambda_0 \Id)^{-1} e^{\tau U} (T-\lambda_0 \Id) \] 
form a $C_0$-semigroup. For $v\in \Dom(U) $, since $T$ and $U$ strongly commute, $Th-\lambda_0 v \in \Dom(U)$ and \begin{align*} \lim_{\tau \searrow 0} \frac{Z_\tau v - v}{\tau} & = (T-\lambda_0\Id)^{-1} \lim_{\tau \searrow 0} \frac{e^{ \tau U} (T-\lambda_0 \Id)v- (T-\lambda_0 \Id)v}{\tau} \\
& = (T-\lambda_0\Id)^{-1} U (Tv-\lambda_0v) = Uv.
\end{align*}
By the uniqueness of the infinitesimal generator we conclude that $Z_\tau = e^{\tau U}$. In other words, 
\[ e^{\tau U} T = T e^{\tau U}
\]
for all $\tau>0$, as we wished to show.
\end{proof}

\begin{proof}[Proof of Theorem \ref{lifting-thm}]
 Set $A=(A_1\cdots A_d)$,  let $v\in \cH$ and notice that, since
 for $ 1\leq i \leq d$ the operators $A_i$ strongly commute with $A_{d+1}$, 
by \cite[Section 34, Theorem 4(i)]{lax2002functional} we can infer
that  $e^{-i\tau A_{d+1}} A^\alpha v \in \Dom A_{d+1}$. Therefore,
the map $\Theta $ is well-defined. Furthermore,
\begin{align*}
  \|\Theta v\|^2_{\cL^2(\Delta)}&=\int_{\bbN_0^d\times\bbR_{-}}\norm \Theta v(\alpha,\lambda)\norm^2_\Delta  d\mu(\alpha,\lambda)\\
  &=\int_{\bbN_0^d\times\bbR_{-}}\norm e^{-i\lambda A_{d+1}} A^\alpha v\norm_\Delta^2 \frac{|\lambda|^{|\alpha|}}{2^{|\alpha|-1}\alpha!}\, d\alpha d\lambda\\
  &=\int_{\bbN_0^d\times\bbR_{-}}\Im \langle A_{d+1} e^{-i\lambda A_{d+1}} A^\alpha v, e^{-i\lambda A_{d+1}}A^\alpha v \rangle_\cH\, \frac{|\lambda|^{|\alpha|}}{2^{|\alpha|-1}\alpha!}\,d\alpha d\lambda\\\
  &\qquad-\int_{\bbN_0^d\times\bbR_{-}} \sum_{i=1}^d \langle A^{\alpha+e_i} e^{-i\lambda A_{d+1}}v, A^{\alpha +e_i}e^{-i\lambda A_{d+1}}v \rangle_{\cH}\, \frac{|\lambda|^{|\alpha|}}{2^{|\alpha|+1}\alpha!}\,d\alpha d\lambda.
 \end{align*}
Now we observe that for $v\in \Dom A_{d+1}$ we have
\begin{align*} 
 \partial_\lambda\|e^{-i\lambda A_{d+1}}v\|^2_{\cH} & =\langle -iA_{d+1}e^{-i\lambda A_{d+1}}v, e^{-i\lambda A_{d+1}}v\rangle_{\cH}+\langle e^{-i\lambda A_{d+1} }v,-i A_{d+1}e^{-i\lambda A_{d+1}}v\rangle_{\cH}\\
 &= 2\Im \langle A_{d+1} e^{-i\lambda A_{d+1}}  v, e^{-i\lambda A_{d+1}}v \rangle_\cH.
\end{align*}
From the fact that $(A_1, \dots, A_{d+1})$ is strongly Siegel--dissipative we deduce that $iA_{d+1}+\varepsilon \Id $ is maximal dissipative. In this case the Lumer--Philips Theorem guarantees that $ iA_{d+1}+\varepsilon \Id $ generates a contraction semigroup; in other words, for any $v\in \cH, \tau < 0 $,
\[ e^{-\tau \varepsilon} \norm e^{-i\tau A_{d+1}} v \norm_{\cH} \leq \|v\|_{\cH} . \]
Therefore, $\norm e^{-i\tau A_{d+1}} v \norm_{\cH} $ decays exponentially and we can integrate by parts as follows,
\begin{align*}
 \|\Theta v\|^2_{\cL^2(\Delta)}&=\sum_{\alpha\in \bbN_0^d}\int_{-\infty}^0 \partial_\lambda \|e^{-i\lambda A_{d+1}} A^\alpha v\|^2_{\cH}\, \frac{|\lambda|^{|\alpha|}}{2^{|\alpha|}\alpha!}\, d\lambda\\
 &\qquad - \sum_{\alpha\in \bbN_0^d}\int_{-\infty}^0 \sum_{i=1}^d \|A^{\alpha +e_i}e^{-i\lambda A_{d+1}}v\|^2_{\cH}\, \frac{|\lambda|^{|\alpha|}}{2^{|\alpha|+1}\alpha!}d\lambda\\
 &=\|v\|^2_{\cH}+\sum_{\alpha\in \bbN_0^d \setminus \{ 0 \}}\int_{-\infty}^0\partial_\lambda \|e^{-i\lambda A_{d+1}} 2^{-\frac{|\alpha|}{2}}A^\alpha v\|^2_{\cH}\, \frac{|\lambda|^{|\alpha|}}{\alpha!}d\lambda\\
  &\qquad - \sum_{i=1}^d \sum_{\alpha\in \bbN_0^d }\int_{-\infty}^0  \|e^{-i\lambda A_{d+1}} 2^{-\frac{|\alpha|+1}{2}}A^{\alpha + e_i }v\|^2_{\cH}\, \frac{|\lambda|^{|\alpha|}}{\alpha!}d\lambda\\
  &=\|v\|^2_{\cH}+\sum_{\alpha\in \bbN_0^d \setminus \{ 0 \}}\int_{-\infty}^0 \|e^{-i\lambda A_{d+1}}2^{-\frac{|\alpha|}{2}}A^\alpha v\|_\cH^2\frac{|\alpha||\lambda|^{|\alpha|-1}}{\alpha!}\, d\lambda\\
  & \qquad - \sum_{\alpha\in \bbN_0^d }\int_{-\infty}^0 \sum_{i=1}^d \|e^{-i\lambda A_{d+1}} 2^{-\frac{|\alpha|+1}{2}}A^{\alpha + e_i }v\|^2_{\cH}\, \frac{|\lambda|^{|\alpha|}}{\alpha!}d\lambda\\
  &=\|v\|^2_{\cH}.
\end{align*}
To see why the two sums cancel each other out, notice that 
\[ \|e^{-i\lambda A_{d+1}} 2^{-\frac{|\alpha|+1}{2}}A^{\alpha + e_i }v\|^2_{\cH}\, \frac{|\lambda|^{|\alpha|}}{\alpha!} = \|e^{-i\lambda A_{d+1}} 2^{-\frac{|\alpha+e_i|}{2}}A^{\alpha + e_i }v\|^2_{\cH}\, \frac{(\alpha_i+1)|\lambda|^{|\alpha+e_i|-1}}{(\alpha+e_i)!}  \] 
and use the combinatorial identity 
\[ \sum_{i=1}^d\sum_{\alpha\in \bbN_0^d } c_{\alpha+e_i} (\alpha_i+1) = \sum_{i=1}^{d}\sum_{\beta\, : \, \beta_i \geq 1 } c_\beta \beta_i = \sum_{\beta\in \bbN_0^d\backslash\{0\}} c_\beta \sum_{i\,:\, \beta_i \geq 1 } \beta_i = \sum_{\beta \in \bbN_0^d\backslash\{0\}} c_\beta |\beta|. \]

We now lift the operator $S_{\gamma,\tau}^*$ to the space $\cL^2(\Delta)$ by tensoring with the identity. 
Explicitly, for $ g : \bbN_0^d \to \cH_\Delta $, 
\[ [ S_{\gamma,\tau }^* \otimes \Id  ]g  (\alpha, \tau) =  \frac{|\lambda+\tau|^{d-|\alpha|-|\gamma|}}{|\lambda|^{d-|\alpha|}} \frac{(\alpha+\gamma)!}{\alpha ! } 2^{|\gamma|} g(\alpha+\gamma, \lambda + \tau ).   \]

We now show that the operator $\Theta$ is an intertwining operator for
the couple $S^*_{\gamma,\tau}\otimes \Id $ and $e^{-i\tau A_{d+1}}
A^\gamma$ where $A=(A_1,\dots,A_d)$.  For, we have
\[ \Theta ( e^{-i\tau A_{d+1}} A^\gamma v ) (\alpha,\lambda ) = \frac{|\lambda|^{|\alpha|-d}}{\alpha!2^{|\alpha|-\frac12}}  e^{-i(\tau+\lambda) A_{d+1}} A^{\alpha+\gamma} v , \]
whereas, on the other hand,
\begin{align*}
   ( [S^*_{\gamma,\tau}\otimes \Id] \Theta v ) &(\alpha,\lambda)  = \frac{|\lambda+\tau|^{d-|\alpha|-|\gamma|}}{|\lambda|^{d-|\alpha|}} \frac{(\alpha+\gamma)!}{\alpha !  } 2^{|\gamma|} \Theta v (\alpha+\gamma, \lambda + \tau ) \\
   & =  \frac{|\lambda+\tau|^{d-|\alpha|-|\gamma|}}{|\lambda|^{d-|\alpha|}} \frac{(\alpha+\gamma)!}{\alpha !  } 2^{|\gamma|} \frac{|\lambda+\tau|^{|\alpha|+|\gamma|-d}}{(\alpha+\gamma)! 2^{|\alpha|+|\gamma|-\frac12}} e^{-i(\tau+\lambda) A_{d+1}} A^{\alpha+\gamma} v \\
   & = \frac{|\lambda|^{|\alpha|-d}}{\alpha!2^{|\alpha|-\frac12}}  e^{-i(\tau+\lambda) A_{d+1}} A^{\alpha+\gamma} v \\
   & =  \Theta ( e^{-i\tau A_{d+1}} A^\gamma v ) (\alpha,\lambda ).
\end{align*} 
In conclusion, the diagram \eqref{diagram-Siegel} commutes and $(ii)$ is proved.
\end{proof}

Finally, we prove our von Neumann type inequality.

\begin{proof}[Proof of Theorem \ref{thm-vonNeumann}]
First suppose that the tuple $(A_1,\dots A_{d+1}) $ is strongly
Siegel-dissipative. 
Let $p(z)=\sum_{|\gamma|\le N} c_\gamma z^\gamma$ be a polynomial in
$z=(z_1,\dots,z_d)$. Let
$(\tau_1,\dots,\tau_d)\in\bbR_-^d$ be fixed and let $M_j$ be the
operator on $\cH$ given by $M_j=e^{-i\tau_j A_{d+1}}A_j$, $j=1,\dots,d$. Then,
$$
p(M_1,\dots,M_d)= \sum_{|\gamma|\le N} c_\gamma
e^{-i\big(\sum_{j=1}^d\gamma_j\tau_j\big)A_{d+1}}A^\gamma =:
\sum_{|\gamma|\le N} c_\gamma e^{-i\tau_\gamma A_{d+1}} A^\gamma ,
$$
so that
\begin{align*}
 \Theta p(M_1,\dots,M_d) & =   \sum_{|\gamma|\le N} c_\gamma \Theta
e^{-i\tau_\gamma A_{d+1}}A^{\gamma}\\
&=  \sum_{|\gamma|\le N} c_\gamma [S_{\gamma,\tau_\gamma}^*\otimes \Id ]\Theta\\
& = \Big(\Big[ \sum_{|\gamma|\le N} c_\gamma  S_{\gamma,\tau_\gamma}^* \Big] \otimes \Id\Big)\Theta .
\end{align*}
Letting $m_j$ be the multiplier on $DA$ given by $m_j(\z,\z_{d+1})=
\z_je^{-i\tau_j\z_{d+1}}$, $j=1,\dots,d$,  we have
\begin{align*}
\big\Vert  p(M_1,\dots,M_d)(v) \big\Vert_\cH
& = \big\Vert \Theta p(M_1,\dots,M_d) (v)\big\Vert_{\cL^2(\Delta)} \\
& = \big\Vert \Big(\Big[ \sum_{|\gamma|\le N} c_\gamma
S_{\gamma,\tau_\gamma}^* \Big] \otimes \Id\Big)\Theta (v)\big\Vert_{\cL^2(\Delta)} \\
  & \le \Big\Vert \Big[ \sum_{|\gamma|\le N} c_\gamma
    S_{\gamma,\tau_\gamma}^*  \Big] \otimes \Id  \Big\Vert_{\cB(\cL^2(\Delta))} 
    \| \Theta (v)\|_{\cL^2(\Delta)} \\
  & = \Big\Vert  \sum_{|\gamma|\le N} c_\gamma
    S_{\gamma,\tau_\gamma}^*   \Big\Vert_{\cB(L^2(\bbN_0\times\bbR_-,d\mu))} \|v\|_\cH,
\end{align*}
so that
\begin{align*}
  \big\Vert  p(M_1,\dots,M_d) \big\Vert_{\cB(\cH)}
  & \le \Big\Vert  \sum_{|\gamma|\le N} c_\gamma
    S_{\gamma,\tau_\gamma}^*
    \Big\Vert_{\cB(L^2(\bbN_0\times\bbR_-,d\mu))}\\
  & = \Big\Vert  \sum_{|\gamma|\le N} \overline{c_\gamma}
    S_{\gamma,\tau_\gamma}
    \Big\Vert_{\cB(L^2(\bbN_0\times\bbR_-,d\mu))}\\
  & =
  \Big\Vert \sum_{|\gamma|\le N}  c_\gamma  \z^{\gamma}
  e^{-i\tau_\gamma \z_{d+1} }
  \Big\Vert _{\cM(DA)} \\
  & =\big\| p(m_1,\dots,m_d)\big\|_ {\cM(DA)}.
\end{align*}

In the general case let $ (A_1, \dots A_{d+1})$ be Siegel-dissipative. Then, if we replace $A_{d+1}$ by $A_{d+1}+i\varepsilon \Id $ we get a strongly  Siegel-dissipative tuple of operators. Applying the von Neumann type inequality we have 
\begin{equation*}
\| p( e^{\varepsilon\tau_1} M_1 , \dots, e^{\varepsilon\tau_k}M_d  ) \|_{\cB(\cH)}\leq \|p(m_1,\dots, m_d)  \|_{\cM(DA)},
\end{equation*}
However, the right hand side of the inequality does not depend on $\varepsilon$ and the left hand side converges in the operator norm as $\varepsilon
 \to 0 ^+$. In fact, it suffices to prove the convergence for each term of the polynomial separately; we have
 \begin{equation*} \norm e^{-i\tau A_{d+1}}e^{\tau \varepsilon} -  e^{-i\tau A_{d+1}} \norm  = \norm  e^{-i\tau A_{d+1}} \norm (1-e^{\tau\varepsilon} ) 
  \leq (1-e^{\tau\varepsilon}) \overset{\tau \nearrow
    0}{\longrightarrow} 0.
\end{equation*}
 Therefore, we can pass to the limit and obtain the desired inequality.
 \end{proof}

\medskip

\emph{Acknowledgments.}\, We would like to thank the anonymous referee for his/her helpful comments.

\ms 

\bibliography{vonNeumann-bib}
 \bibliographystyle{amsalpha}

\end{document}